\documentclass[12pt,reqno,twoside]{amsart}

\usepackage{amsmath}

\usepackage{amssymb}

\usepackage{epsfig}

\usepackage{color}

\numberwithin{equation}{section}

\newif\ifdraft\drafttrue
\draftfalse

\long\def\comliorryan#1{\ifdraft{\marginpar{\sn
#1 \ LF \& RB}}\else\ignorespaces\fi}

\newcommand{\bq}{{\mathbb{Q}}}

\newcommand{\bn}{{\mathbb{N}}}
\newcommand{\br}{{\mathbb{R}}}
\newcommand{\bz}{{\mathbb{Z}}}

\newcommand{\ct}{\mathcal T}

\font\sb = cmbx8 scaled \magstep0

\font\sn = cmssi8 scaled \magstep0

\long\def\comdima#1{\ifdraft{\marginpar{\sb
#1 \ DK}}\else\ignorespaces\fi}

\newcommand\ba{badly approximable}

\newcommand\da{Diophantine approximation}
\newcommand\di{Diophantine}

\newcommand\ssm{\smallsetminus}

\newcommand\eq[2]{{\ifdraft{\ \tt [#1]}\else\ignorespaces\fi}\begin{equation}\label{eq:#1}{#2}\end{equation}}

\newcommand {\equ}[1]     {\eqref{eq:#1}}
\newcommand{\under}[2]{\underset{\text{#1}}{#2}}

\newcommand{\Q}{{\mathbb {Q}}}

\newcommand{\R}{{\mathbb{R}}}

\newcommand{\T}{{\mathbb{T}}}

\newcommand{\Z}{{\mathbb{Z}}}

\newcommand{\N}{{\mathbb{N}}}

\newcommand{\supp}{\operatorname{supp}}

\newcommand {\ignore}[1]  {}

\newcommand{\df}{{\, \stackrel{\mathrm{def}}{=}\, }}

\newcommand{\vre}{\varepsilon}

\newcommand\hd{Hausdorff dimension}

\newcommand{\cy}{{\mathcal Y}}
\newcommand{\cf}{{\mathcal F}}

\newtheorem{thm}{Theorem}[section]
\newtheorem{lem}[thm]{Lemma}
\newtheorem{prop}[thm]{Proposition}
\newtheorem{cor}[thm]{Corollary}

\newtheorem{defn}[thm]{Definition}

\title[Schmidt's game, fractals, and  numbers normal to no base]{Schmidt's game, fractals, and \\ numbers normal to no base}
\author[Broderick, Bugeaud, Fishman, Kleinbock and Weiss]{Ryan Broderick, Yann Bugeaud, Lior Fishman, \\ Dmitry Kleinbock and Barak Weiss\\
} %%d removed preliminary

\address{Brandeis University, Waltham MA 02454-9110, USA, 
\newline
 {\tt ryanb@brandeis.edu}, {\tt lfishman@brandeis.edu}, {\tt kleinboc@brandeis.edu}}

\address{Ben Gurion University, Be'er Sheva, Israel 84105 {\tt barakw@math.bgu.ac.il}}

\address{Universit\'e de Strasbourg, 67084 Strasbourg, France {\tt bugeaud@math.u-strasbg.fr}}

\date{January 2010}

\begin{document}

 \begin{abstract} 
Given $b > 1$ and  $y \in  \R/\Z$,
we consider the set of $x\in \R$
such that $y$ is not a limit point of the sequence $\{b^n x\,\bmod 1: n\in\N\}$.
Such sets are known to have full \hd, and in many cases have been shown to
have a stronger property of being winning in the sense of Schmidt.
In this paper, by utilizing Schmidt games,  we prove that  these sets  
 and their bi-Lipschitz images must intersect
with `sufficiently regular'  fractals $K\subset \R$  (that is, supporting measures  $\mu$
satisfying certain decay conditions).
Furthermore, the intersection has full dimension in $K$ if $\mu$ satisfies a power law (this holds for
example if $K$ is the middle third Cantor set).  %%y Better, not to name the Cantor set
%As a direct consequence of our theorems
Thus it follows that the 
set of numbers in the middle third Cantor set which are normal to no %%y 
base has dimension $\log2/\log3$.

\end{abstract}
\maketitle

\section{Introduction} \label{intro} 
Let $b \ge 2$ be an integer.
A real number $x$ is said to be {\sl normal\/} to base $b$ if,
for every $n\in\bn$, every block of $n$ digits from 
$\{0, 1, \ldots , b-1\}$ occurs in the base-$b$ expansion
of $x$ with 
asymptotic frequency $1/b^n$. 
%Without loss of generality we can restrict $x$
%to lie in $[0,1]$, identifying the latter with $ S^1 \cong\br/\bz$. Then $x$ is normal to base $b$
Equivalently, let $f_{b}$ be 
the self-map of $\T\df \R/\Z$ given by $x\mapsto bx$, and denote by $\pi: {x} \to x \,\bmod 1
$
the natural projection $\R\to\T$. Then $x$ is normal to base $b$  iff for any interval $I\subset \T$ with $b$-ary rational endpoints one has
$$\lim_{n\to\infty}\tfrac1n\#\big\{0 \le k \le n-1 : f^k_b\big( \pi({x})\big)\in I\big\} = \lambda(I)\,,$$ 
where $\lambda$ stands for Lebesgue measure.
\'E.\ Borel %%y
established that $\lambda$-almost all numbers  are normal to every integer base;
clearly this is also a consequence of Birkhoff's Ergodic Theorem and the ergodicity of 
$(\T,\lambda, f_{b})$.

%We remark\comdima{Not sure if we need to remark this.}  that one does not know a single explicit example of a number
%having the above property. 
%On the other hand 
Note that it is easy to exhibit many non-normal numbers in a given base $b$.
For example, denote by $E_b$ the set of  real numbers with a uniform upper bound on the 
number of consecutive zeroes  in their base-$b$ expansion.
Clearly those are not normal, and it is not hard to show that the  \hd\ of $E_b$ is equal to $1$. 
Furthermore, it was shown by W.\ Schmidt \cite{S1} that for any $b$ and any $0 < \alpha < 1/2$, the set $E_b$  is an 
 {\sl $\alpha$-winning set\/}
 of a game which later became known as Schmidt's game. 
 This property implies full \hd\ but is considerably stronger; for example, an intersection of  countably many $\alpha$-winning sets is also  $\alpha$-winning (we describe the definition and features of  Schmidt's game
in \S  \ref{Schmidt}). %In view of the countable intersection property of those sets, 
Thus it follows that  the set of real numbers  $x$ such that for each $b\in\Z_{\ge 2}$
their base-$b$ expansion does not contain more than $C = C(x,b)$ consecutive 
zeroes
%having bounded strings of zeroes in every integer base
has full Hausdorff dimension. Obviously, such numbers are normal to
no base. 

Now fix $y \in \T$ and a map $f:\T\to \T$, 
and,
following notation introduced in \cite{K}, %$f:\R\to \R$, 
consider
\eq{defefu}{
 E(f,y)\df \big\{ x\in\T: y \notin \overline{\{f^n(x) : n\in \N\}}\big\} \,,
}
the set of points with $f$-orbits staying away from $y$.
For brevity we will write %$E(b,y)$ for $E(f_b,y)$. 
\eq{defebu}{
 E(b,y) = \big\{ x\in\T: y \notin \overline{\{f_b^n(x) : n\in \N\}}\big\} \,.
} 
for $E(f_b,y)$. 
%by $E(b,y)$  the set of $x\in \T$
%whose orbit closure under $M_{b}$ does not contain $y$. 
Obviously $E(b,0)$ is a subset of $\pi(E_b)$
for any $b$. It is known that  $\dim\big(E(b,y)\big) = 1$
for any $b$ and any $y\in \T$,
see e.g.\ \cite{Urbanski, Do}. Moreover, these sets\footnote{The results of \cite{Urbanski, Do, Ts},
are more general, with $f_b$ replaced by an arbitrary sufficiently smooth expanding self-map of $\T$.}
have been recently proved by J.\ Tseng \cite{Ts} to be $\alpha$-winning, where $\alpha$ is independent of $y$ but (quite badly) depends on $b$. 
In particular, it follows that for any {\it bounded\/} sequence $b_1,b_2,\ldots\in\bz_{\ge 2}$ and any $y_1,y_2,\ldots\in \T$, one has
\eq{finite}{
\dim\left(\bigcap_{k = 1}^\infty E({b_k},y_k)\right) = 1\,.}
Another related result is that of S.G.\ Dani \cite{D}, who proved that for any $y\in \bq/\bz$ and any $b\in\bz_{\ge 2}$, 
the sets $E(b,y)$ are $\frac12$-winning (in fact, his set-up is more general and involves semisimple
endomorphisms of the $d$-dimensional torus). Consequently, %an analogue of  
\equ{finite} holds with no upper bound on $b_k$
%with finite intersection  replaced by a countable one 
as long as points $y_k$ are chosen to be
rational (that is, pre-periodic for maps $f_{b}$).

The main purposes %%y 
of the present note are to extend \equ{finite} by %allow countable intersections,
removing an upper bound\footnote{After this paper was finished we learned of
an alternative approach \cite{Fae, FPS} showing that sets
$E(b,y)$ are $\frac14$-winning for any $y\in\T$ and any $b\in\bz_{\ge 2}$; also, 
 in a sequel \cite{BFK} to this paper it is explained that $\frac14$ can be replaced by $\frac12$.}
 on $b_k$,  
and to consider %%y 
intersections with certain fractal subsets of $\T$ such as e.g.\
 the middle third Cantor set. %%y
 In fact it will be convenient to lift the problem from $\T$ to $\R$ and 
 work with $ \pi^{-1}\big(E(b,y)\big)$;
in other words, %take an arbitrary lift $F:\R\to\R$ of $f_b$ and 
consider
\eq{defebut} {
\tilde E(b,y)\df \big\{ x\in\R: y \notin \overline{\{\pi( b^nx) : n\in \N\}}\big\} % \,.
} 
%where $\tilde f$ is any lift of $f$ to a self-map of $\R$ 
Clearly this set %does not depend on the choice of the lift
%When $b \in \Z$, t
%and 
is periodic (with period $1$); however we are going to study its intersections
with (not necessarily periodic)  subsets $K\subset \R$, for example, with their bi-Lipschitz images.
Another  advantage of switching from \equ{defebu} to \equ{defebut} is that  the latter makes
sense even
 when %$F$ is not a lift of a circle map, for example, when $F(x) = bx$ with
 $b > 1$ is not an integer\footnote{To make sense of \equ{defebu} when $b\notin \Z$ some efforts
 are required, see \S \ref{nonlinear}.}. 
 %For this choice of $F$ we will use the notation $\tilde E(b,y)$ in place of
  %$\tilde E(F,y)$, that is, put
This set-up has been extensively studied; 
for example A.\ Pollington proved in  \cite{P1} that the intersection $\bigcap_{k = 1}^\infty\tilde E(b_k,y_k)$
has \hd\ at least $1/2$ for any choices of $y_k\in\T$ and $b_k > 1$, $k\in\N$.
More generally,  there are similar results with $(b^n)$ in \equ{defebut} % being 
replaced by an arbitrary {\sl lacunary\/} sequence $\ct =  (t_n)$ of positive real numbers (recall that $\ct$ is called lacunary if
$\inf_{n\in \N}\frac{t_{n+1}}{t_n} > 1$). Namely,  generalizing \equ{defebut}, fix $\ct$ as above and
a sequence $\cy = (y_n)$ of points in $\T$, and
%and for  a sequence 
%$\ct =  (t_n)$ %of positive real numbers  
%one can 
define 
%%y removed the tag
%\eq{defetyt}{
$$\tilde E(\ct,\cy)\df \big\{ x\in\R: \inf_{n\in \N} d\big(\pi(t_nx) , y_n\big) > 0\big\} \,.$$
%}
Here and hereafter $d$ stands for the usual distance on $\T$ or $\R$.
We will write $\tilde E(\ct,y)$ when $\cy = (y)$ is a constant sequence, that is,
%\eq{defetut}{
$$\tilde E(\ct,y)\df \big\{ x\in\R: y \notin \overline{\{\pi(t_nx) : n\in \N\}}\big\} \,.$$
%}
 It is a result of Pollington \cite{P2} and B.\ de Mathan \cite{Ma} that
the sets $\tilde E(\ct,0)$ have \hd\ $1$ for any lacunary sequence $\ct$; %%y
see also \cite[Theorem 3]{BHKV} for a multi-dimensional generalization. %%y
%%d made the reference more precise
Moreover, one can show, as mentioned by N.\ Moshchevitin in \cite{M1},
 that those sets are  $\frac12$-winning.  
 %%y \comdima{any other references on lacunary sequences?}

Our main theorem %strengthens 
extends the aforementioned results %and tilts them 
in several directions. 
\ignore{We are able to compose the maps $x\mapsto t_nx$ with translations,
that is, consider a sequence of  similarities $\cf = (F_n)$ of $\R$, where $F_n = t_n x + c_n$.
With some abuse of notation, let us say that $\cf$ is lacunary if so is $\ct$. 
We will refer to $t_n$ as the {\sl scaling constant\/} of $F_n$. Generalizing \equ{defetut},
for a sequence $\cf = (F_n)$ of self-maps of $\R$ we are going to look at
%\comdima{I guess there are too many definitions in the introduction now, need to find a way to cut down..}
\eq{defefyt}{
\tilde E(\cf,y)\df \big\{ x\in\R: y \notin \overline{\{\pi\big(F_n(x)\big) : n\in \N\}}\big\} \,.
}
%{\sb(or what's a better way to say it?)} -- investigating 
Also, }
We will  allow arbitrary sequences $\cy$, and will study  intersection  of sets $\tilde E(\ct,\cy)$ with certain
fractals $K\subset \R$.
 Namely, if $K$ is a closed subset of the real line, following \cite{F}, we 
 will play Schmidt's game on the metric space $K$ with the   induced metric. We will say that a subset $S$ of $\R$
 is {\sl $\alpha$-winning on $K$\/}  if $S\cap K$ is an  $\alpha$-winning set for the game played on $K$. 
 See \S  \ref{Schmidt} for more detail. Further, in \S \ref{measures} we define
and discuss so-called {\sl $(C,\gamma)$-absolutely decaying\/} 
measures -- a %variant of a 
notion introduced in \cite{KLW}.  %%y
% which first appeared
%in this form in \cite{PV}. 
Here is our main result:

\begin{thm}
\label{theorem}
%\Kaf\
Let $K$ be the support of a $(C,\gamma)$-absolutely decaying measure on $\R$,
and let \eq{alpha}{\alpha \le \frac14\left( \frac{1}{3C} \right)^{\frac{1}{\gamma}}.}
Then for every bi-Lipschitz map $\varphi: \R \to \R$, any sequence $\cy$ of points  in $\T$,
and any lacunary sequence  $\ct$, the set  %%y
% with scaling constants forming a lacunary sequence, 
$\varphi\big(\tilde E(\ct,\cy)\big)$ is $\alpha$-winning on $K$.
\end{thm}

%This, in particular, applies to $\cf = (f^n)$ where $f(x) = bx + c$ with $b > 1$.
We  also show in \S\ref{Schmidt} %%d
that when $K$ is %%y
as in the above theorem and $S$ is
winning on $K$, one has $\dim(S\cap K) \ge \gamma$. Furthermore,  $\dim(S\cap K)  =  \dim( K) 
$ if $\mu$ satisfies a power law. 
Consequently, in view of the countable intersection property of winning sets, 
for any choice of lacunary sequences $\ct_k$, sequences $\cy_k$ of points in $\T$, and
bi-Lipschitz maps $\varphi_k: \R\to \R$,
one has
\eq{infinite}{
\dim\left(K \cap \bigcap_{k = 1}^\infty \varphi_k\big( \tilde E(\ct_k, \cy_k)\big)\right) \ge \gamma\,,}
where $\gamma$ is as in Theorem \ref{theorem} (see Corollary \ref{cor}). 
Thus  on any $K$ as above it is possible to find a set of positive \hd\ consisting of numbers
which are normal to no base. 

%On the other hand, it is clear from the definitions in \S\ref{measures} that 
%absolute decay of a measure can be checked by looking at small enough balls. 
%is a local condition. 
%In particular one can talk about a 
%subset $K$ of $\T$ supporting a $(C,\gamma)$-absolutely decaying measure. 

Another consequence of the generality of Theorem \ref{theorem} is a possibility to consider orbits of affine 
expanding  maps of the circle, that is,  \eq{deffbc}{f_{b,c}: x\mapsto  bx + c\,,\text{  where }b\in\Z_{\ge 2}\text{  and }c\in \T\,.}
%of 
%Then,
%returning to the original set-up of considering self-maps of $\T$, one can  look at 
%orbits of $f_{b,c}: \T\to \T$ given by $f_{b,c}(x) = bx + c$. 
It then follows that whenever 
$K$, $\alpha$ and $\varphi$ are %%y 
as in Theorem \ref{theorem} and $y\in\T$, 
%%y  O.K. %\comdima{Do we need to formally prove it later or it is clear?}
%\eq{circle}{
%\dim\left(K \cap \bigcap_{k = 1}^\infty \varphi_k\big( E(f_{b_k,c_k}, y_k)\big)\right) \ge \gamma
 the set $\varphi\left(\pi^{-1}\big( E(f_{b,c}, y)\big)\right)$ is $\alpha$-winning on $ K$ (see Corollary \ref{affine}).
 %for any  $b\in\Z_{\ge 2}$, $c,y\in \T$, and bi-Lipschitz
 %$\varphi: \T\to \T$. 
 In particular, 
 %%y  \comdima{Is the last statement clear?}   Yes, I think!
 $E(f_{b,c}, y)$ itself   is  $\alpha$-winning on any subset of $\T$ supporting a measure
 which can be lifted to a $(C,\gamma)$-absolutely decaying measure on $\R$.

Also, as is essentially proved in \cite{F}, a bi-Lipschitz 
image of the set %\eq{def ba}{
$$\bold{BA} \df \left\{{x}\in\br : \exists \,c = c({x}) >0\text{ s.\ t.\ }\left|{x} - \frac{p}{q}\right| > \frac{c}{q^2}\ \ 
\forall(p,q) \in \mathbb{Z}\times \N\right\}$$
%}
 of {\sl badly approximable\/} numbers is also $\alpha$-winning on $K$ under
the same assumptions on $K$ %\comdima{Added references.} 
(see also \cite{KW1, KTV}). We discuss this in \S\ref{proof} (see Theorem \ref{ba}). 
Thus the intersection of the set in the left hand side of \equ{infinite} with 
$\varphi(\bold{BA})$, where $\varphi: \R \to \R$ is bi-Lipschitz, will still have \hd\ at least $\gamma$. 
This significantly generalizes V.\ Jarn\'ik's \cite{J} result on the full \hd\ of $\bold{BA}$, as well
as its strengthening by Schmidt \cite{S1}.
Note that $\bold{BA}$ is a nonlinear analogue of $\tilde E(b,0)$, with $f_{b}$ replaced by the Gauss map; %\comdima{We should a remark about it at the end of the paper.} 
this
naturally raises a question of extending our results to more general self-maps of $\T$, see \S\ref{nonlinear}.

%\smallskip
%%y begin modif.

As a straightforward consequence of our results, we get

\begin{cor} 
Given $K \subset \R$ supporting an absolutely decaying measure $\mu$,
the set of real numbers $x\in K$ that are badly approximable and 
such that, for every $b \ge 2$, their base-$b$ expansion
does not contain more than $C(x, b)$ consecutive identical digits, 
has positive Hausdorff dimension. In particular,
if $\mu$ satisfies a power law (for example, if $K$ is the middle third Cantor set), 
then the dimension
of this set is full.
\end{cor}

%%y end modif.

%\smallskip

The structure of the paper is as follows.
In \S\ref{measures} we describe the class of absolutely decaying measures on $\R$, giving
examples and highlighting the connections between absolute decay and other
properties.  In \S\ref{Schmidt} 
we discuss Schmidt's game
played on arbitrary metric spaces $X$, and then specialize to the case when $X = K$ is a 
subset of $\R$ supporting an absolutely decaying measure.
% which allow one to play Schmidt's game on their supports and imply lower estimates on the \hd\ of winning sets.
%\comdima{Maybe we should also state some corollaries on digit expansions; can add those later.}
Then in \S\ref{proof} we prove the main theorem.
%, and in \S\ref{dim} discuss \hd\ estimates of
%sets which are winning on $K$ as above, and derive Corollary \ref{cor}. 
%\comliorryan{Maybe this final section should also mention our multi-dimensional
%analog to appear in the [BFK] paper.}\comdima{Agree.}
The last section
is devoted to some extensions of our main result %, its applications to studying base-$b$ %%y 
%expansions of real numbers,   %%d
and further open questions.

\smallskip

{\bf Acknowledegments:} Yann Bugeaud would like to thank
 Ben Gurion University at Beer-Sheva, where part of this 
work has been done. This research 
was supported in part by  NSF
grant DMS-0801064, ISF grant 584/04 and  BSF grant  2004149.
%\comdima{Feel free to add BSF.}

%\vfil\eject

\section{Absolutely decaying measures }
\label{measures}

The next definition  %first appeared in \cite{PV}
%and 
describes a property of measures first introduced in \cite{KLW}. 
In this paper
we only consider measures on the real line; however see \S\ref{matrices} for a situation in higher dimensions. In what follows, we denote
by $B(x,\rho)$ the closed ball in a metric space $(X,d)$ centered at $x$ of radius $\rho$,
\eq{ballsdef}{   %%y there were two labels "balls"
B(x,\rho)
 \df \{ y \in X : d(x,y) \le \rho\}\,.} 
 %We will  write $B(x,\rho)$ when the choice of the metric space is clear from the context.
%As was mentioned in the introduction, friendly measures are extensively 
%studied and utilized  in recent research. In what follows we give the definitions
%of a measure being absolutely friendly in the one dimensional case and the 
%interested reader should consult \cite{KLW} and \cite{PV} for the multi-dimensional
%definition.  
\begin{defn}
\label{decay}
Let $\mu$ be a locally finite
Borel measure on $\br$, and let $C, \gamma > 0$. 
We say that $\mu$ is
{\sl $(C,\gamma)$-absolutely decaying\/} if there exists $\rho_{0} >0$ such that for all 
$0< \rho\leq \rho_{0}$,  $x\in  \supp\mu$, $y\in \R$ and $\vre >0$,
\eq{axiom}{
\mu\big(B(x,\rho)\cap B(y,\vre \rho)\big) < C\vre^{\gamma}\mu\big(B(x,\rho)\big)\,.
}
We say $\mu$ is {\sl absolutely decaying\/} if it 
is $(C,\gamma)$-absolutely decaying for some positive $C, \gamma$.
\end{defn}

Many examples of measures satisfying this property are constructed\footnote{The terminology in \cite{KLW} is slightly different; %%y
there, %%y
$\mu$ is called absolutely decaying if  $\mu$-almost every point has a neighborhood $U$
such that the restriction of $\mu$ to $U$ is $(C,\gamma)$-absolutely decaying for some $C,\gamma$;
however in all  examples considered  in \cite{KLW, KW1} a stronger uniform property is in fact established.} in \cite{KLW, KW1}. 
For example, limit measures of finite 
%irreducible 
systems of contracting similarities \cite[\S8]{KLW}
satisfying the
open set condition and without a global fixed point are absolutely decaying. 
See also \cite{U2, U3, U4, SU} for other examples. 
%Note that without loss of generality one can take $\rho_0 < 1/4$ and apply Definition \ref{decay}
%verbatim to measures on $S^1$, by identifying balls in $S^1$ with subintervals of $[0,1]$.

\smallskip

In what follows we highlight the connections between absolute decay and other conditions
 introduced earlier in the literature. 

 \begin{defn}
\label{decay-old}
Let $\mu$ be a locally finite
Borel measure on a metric space $X$. 
%If $x \in \supp \mu$,  $\rho > 0$ and $0 < a < 1$, let us denote
%by $\varkappa(x,\rho,a)$ the ratio ${\mu\big(B(x,\vre\rho)\big) /  \mu\big(B(x ,\rho)\big)}$.
%Then  
One says that $\mu$ is 
{\sl Federer\/} (resp., {\sl efd\/}) if there exists $\rho_{0} >0$ and $0 < \vre,\delta < 1$ such that 
for every $0<\rho\leq \rho_0$
and for any $x\in \supp \mu$,  the ratio \eq{ratio}{\mu\big(B(x,\vre\rho)\big) /  \mu\big(B(x ,\rho)\big)}
is at least (resp., at most) $\delta$.
\ignore{eq{leftside}{
%\begin{aligned}\forall\, 0<\rho\leq \rho_0\text{ and }
%\forall\, x\in \supp \mu\text{ one has}\quad\ \  \\
\mu\big(B(x,a\rho)\big)  > \delta  \mu\big(B(x ,\rho)\big)\,.}
We say that $\mu$ is 
if there exists $\rho_{0} >0$ and $0 < a,\delta < 1$ such that 
for every $0<\rho\leq \rho_0$
and for any $x\in \supp \mu$,  one has
\eq{rightside}{
%\begin{aligned}\forall\, 0<\rho\leq \rho_0\text{ and }
%\forall\, x\in \supp \mu\text{ one has}\quad\ \  \\
\mu\big(B(x,a\rho)\big)  < \delta  \mu\big(B(x ,\rho)\big)\,.}}
\end{defn}
 
 Federer property is usually referred to as  `doubling': see e.g.\ \cite{MU} for discussions and examples.
The term `efd'  (an abbreviation for exponentially fast decay) was introduced by Urbanski; see \cite{U2, U4} for many examples
 and \cite{Veech,W} for other equivalent formulations. The next lemma provides another way to state these properties:
 
\begin{lem} 
\label{equiv} 
Let $\mu$ be a locally finite
Borel measure on a metric space $X$. Then $\mu$ is {\sl Federer\/} (resp., {\sl efd\/}) if and only if
 there exist $\rho_{0} >0$ and $c,\gamma > 0$ such that 
for every $0<\rho\leq \rho_0$, $0 < \vre < 1$, 
and $x\in \supp \mu$,  the ratio \equ{ratio}
is not less (resp., not greater) than $c\vre^{\gamma}$.
\end{lem}

\begin{proof}  The `if' part is clear, one simply needs to choose $\vre$ such that $c\vre^{\gamma} < 1$.
Now suppose $\mu$ is Federer, and  let $\vre_0,\delta$ be such that \eq{Federer}{\mu\big(B(x,\vre_0\rho)\big)  \ge \delta \mu\big(B(x ,\rho)\big)}
for every $0<\rho\leq \rho_0$
and  $x\in \supp \mu$. We are going to put $c =\delta$ and $\gamma= \frac{\log \delta}{\log \vre_0}$.
Take $0 < \vre < 1$, and let $n$ be the largest integer such that
$\vre \leq \vre_0^n$. Then
$$c\vre^{\gamma} = \delta \vre^{\frac{\log \delta}{\log \vre_0}}
= \delta \delta^{\frac{\log \vre}{\log \vre_0}} \leq \delta ^{n+1}\,.$$
Hence
$$c\vre^{\gamma}\mu\big(B(x,\rho)\big) \leq \delta^{n+1}\mu\big(B(x,\rho)\big) 
\underset{\text{\equ{Federer} applied $n$ times}}\leq \mu\big(B(x,\vre_0^{n+1}\rho)\big)
% \leq 
%\mu\big(B(x,\vre\rho)\big)
\,,$$
which, in view of the definition of $n$, implies
%\eq{leftside}
${\mu\big(B(x,\vre\rho)\big) \ge c\vre^{\gamma}\mu\big(B(x,\rho)\big)}$.
Similarly, from the fact that $\mu\big(B(x,\vre_0\rho)\big)  \le \delta \mu\big(B(x ,\rho)\big)$
for every $0<\rho\leq \rho_0$
and  $x\in \supp \mu$ one can deduce the inequality \eq{rightside}{\mu\big(B(x,\vre\rho)\big) \le c\vre^{\gamma}\mu\big(B(x,\rho)\big)}
for every $x,\rho$ and $\vre$, with $c =1/\delta$ and $\gamma= \frac{\log \delta}{\log \vre_0}$.
\end{proof}

Now we can produce an alternative description of absolutely decaying measures on $\R$:
\begin{prop} 
\label{equiv-af} Let $\mu$ be a locally finite
Borel measure on $\br$. Then $\mu$ is absolutely decaying if and only if it is  Federer and efd. \end{prop}

The `if' part is due to Urbanski, see \cite[Lemma 7.1]{U4}; we include 
a proof to make the paper self-contained.  
%%y  \comdima{Or maybe it is better to omit the proof?}
%%y I think one can leave it. If the referee asks to remove it, then we will remove it!

\begin{proof}  Let $\mu$ be $(C,\gamma)$-absolutely decaying, and let $\rho_0$ be as in 
Definition \ref{decay}. Taking $x = y$ and $c = C$ in \equ{axiom} readily  implies \equ{rightside}, i.e.\ the efd property. To show Federer,  
  take  $0<\rho\leq\rho_0$ and $x\in\supp \mu$, and let
 $\vre < 1/4$  satisfy ${C}{\vre^{\gamma}}<{1}/{2}$.
Choose $y_1$ and $y_2$ to be the two distinct points satisfying  $|x-y_i|=(1-\vre)\rho$, $i=1,2$. It clearly follows from Definition \ref{decay} that $\mu$ is non-atomic; 
thus we can write 
$$\mu\big(B(x,\rho)\big) =\mu\big(B(x,\rho)\cap B(y_1,\vre \rho)\big)+\mu\left(B\big(x,(1-2\vre)\rho\big)\right)+\mu\big(B(x,\rho)\cap B(y_2,\vre\rho)\big).$$
Therefore, by \equ{axiom},
$$\mu\big(B(x,\rho)\big) \le \mu\left(B\big(x,(1-2\vre)\rho\big)\right) + 2C\vre^{\gamma}\mu\big(B(x,\rho)\big).$$
Setting $\vre_0 =1-2\vre$ and $\delta=1-2C\vre^{\gamma}$ we get \equ{Federer}.

Conversely, suppose that $\mu$ is both  Federer and efd. In view of Lemma \ref{equiv}, for some $\rho_0 > 0$
and $c_1,c_2,\gamma_1,\gamma_2 > 0$
one has 
%\eq{twopowers}{
$$ c_1\vre^{\gamma_1}\mu\big(B(x,\rho)\big) \le \mu\big(B(x,\vre \rho)\big) \le c_2\vre^{\gamma_2}\mu\big(B(x,\rho)\big)$$
%}
for all
$0< \rho\leq \rho_{0}$,  $x\in \supp \mu$  and $0 < \vre < 1$. Now take 
%$0 <\rho\leq\rho_0$, $x\in\supp \mu$ and 
$\rho < \rho_0/3$ and $y \in B(x,\rho)$.
If $\mu\big(B(x,\rho) \cap B(y,\vre\rho)\big) = 0$, we are done.
Otherwise, there exists $y^{\prime} \in \supp \mu \cap B(y,\vre\rho) \cap B(x,\rho)$. Then
$$\mu(B(x,\rho) \cap B(y,\vre\rho)) \leq \mu(B\big(y^{\prime},2\vre\rho)\big)
 \leq c_2\vre^{\gamma_2}\mu(B\big(y^{\prime},2\rho)\big)$$
 $$
 \leq  c_2\vre^{\gamma_2}\mu\big(B(x,3\rho)\big)
 \leq c_2 c_1^{-1}3^{\gamma_1}\vre^{\gamma_2}\mu\big(B(x,\rho)\big) \,,
 %= C\vre^\gamma \mu(B(x,\rho)).
 $$
which gives \equ{axiom} with $C = c_2c_1^{-1}3^{\gamma_1}$ and  $\gamma = \gamma_2$.
\end{proof}

In particular, suppose that  $\mu$ {\sl satisfies a power law\/}, i.e.\  there exist positive $\gamma, k_1,k_2,\rho_0$
such that for every $x\in \supp \, \mu$ and $0 < \rho<\rho_0$ one has
$$k_1\rho ^{\gamma}\leq\mu\big(B(x,\rho)\big)\leq k_2\rho^{\gamma}\,;
$$
then $\mu$ is clearly efd and Federer, hence absolutely decaying. However there exist
examples of absolutely decaying measures without a power law, see \cite[Example 7.5]{KW1}.
Also, recall that the
{\sl lower pointwise dimension\/} of $\mu$ at $x$ is defined as
$$\underline{d}_\mu(x) \df \liminf_{\rho\to 0} \frac{\log\mu(B(x,\rho))}{\log \rho}\,,
$$
and, for an open $U$ with $\mu(U) > 0$ let
\eq{dmu}{\underline{d}_\mu(U) \df \inf_{x\in \supp\,\mu\cap U}\ \underline{d}_\mu(x) \,.}
Then it is known, see e.g.\ \cite[Proposition 4.9]{Fa}, that \equ{dmu} constitutes a lower bound for the \hd\ of $\supp\,\mu\cap U$ (this bound is sharp when 
$\mu$ satisfies a power law). It is easy to see that
$\underline{d}_\mu(x)\ge \gamma$ for every $x\in \supp \, \mu$ 
whenever $\mu$  is $(C,\gamma)$-absolutely decaying: indeed, let
$\rho_0$ be as in Definition \ref{decay} and take $\rho < \rho_0$ and $x\in \supp \, \mu$; then, letting $\vre =\frac{\rho}{\rho_0}$, one has
$$\mu\big(B(x,\rho)\big) \le C\left(\frac{\rho}{\rho_0}\right)^\gamma\mu\big ( B(x,\rho_0)\big),$$
thus, for $\rho < 1$, %%y  I added one line of computation
$$
\frac{\log \mu\big(B(x,\rho)\big)}{\log \rho} \ge \gamma 
+ \frac{\log C - \gamma \log \rho_0 + \log \mu\big(B(x,\rho_0)\big)}{\log \rho},
$$
and the claim follows.
\smallskip

In the next section we will show that sets supporting  absolutely decaying measures on $\R$
work very well as playing fields for Schmidt's game. The aforementioned lower estimate
for $\underline{d}_\mu(x)$ will be used to provide a lower bound for the \hd\ of winning sets of the game.

\ignore{\medskip

 We remark that  absolutely decaying measures %, and  a less restrictive notion of 
 %{\sl decaying\/} measures, 
 were originally  defined on $\R^n$; namely, a measure on $\R^n$
 is  absolutely decaying if it %is Federer and 
 satisfies \equ{axiom} with balls replaced by neighborhoods of 
  affine hyperplanes. In this generality the analogue of Proposition 
 \ref{equiv-af} \comdima{Find a way to rephrase?}
is false; in fact the Federer condition has to be assumed in order to produce a multi-dimensional
analogue of the main results of this paper. }
 
\section{Schmidt's game}
\label{Schmidt}
 In this section we describe
the game, first 
introduced by  Schmidt in \cite{S1}.
Let $(X,d)$ be a complete metric space.
Consider  $\Omega \df X \times \mathbb{R}_+$, and define a 
partial ordering
\begin{center}
$(x_2,\rho_2)\le_{s}(x_1,\rho_1)$\  if \  $\rho_2+d(x_1,x_2)\le \rho_1$.
\end{center} 
We associate to each pair $(x,\rho)$ a ball in $(X,d)$ via the `ball' function 
$B(\cdot)$ as in \equ{ballsdef}. %%y 
%\begin{center}
%$B_X(x,\rho)\df \{ y \in X : d(x,y) \le \rho\}$.
% \end{center}
% We will simply write $B(x,\rho)$ when the choice of the metric space is clear from the context.
Note that $(x_2,\rho_2)\le_{s}(x_1,\rho_1)$ clearly implies (but is not necessarily implied by)
$B(x_2,\rho_2) \subset B(x_1,\rho_1)$. However the two conditions are equivalent when
$X$ is a Euclidean space. 

%\smallskip

Schmidt's game is played by two players, whom, following a notation used in \cite{KW2}, we will call\footnote{The players were referred to as  `white' and  `black' by Schmidt, and as $A$ and $B$ in some subsequent literature; a suggestion to use the
Alice/Bob  nomenclature is due  to Andrei Zelevinsky.}
 %$\bold W$ and $\bold B$
Alice and Bob. The two players are
equipped with parameters $\alpha$ and $\beta $ 
respectively, satisfying $0<\alpha ,\beta <1$. 
 Choose a subset ${S}$ of $ X$ (a target set).
The game starts with Bob picking $x_1\in X$ and $\rho > 0$, 
hence specifying a pair $\omega_1 = (x_1,\rho)$. Alice 
and Bob then take turns choosing $\omega'_k = (x'_k,\rho'_k)\le_s\omega_k$
and $\omega_{k+1}= (x_{k+1},\rho_{k+1})\le_s\omega'_k$ respectively satisfying 
\eq{balls}{\rho_k' = \alpha \rho_k\text{ and }\rho_{k+1} = \beta \rho_k'\,.}
%may now choose any point $x'_1\in X$ provided that 
%$\omega'_1\df (x'_1,\alpha \rho)\le_{s}\omega_1 $. %\comdima{Reindexed points, hope it is OK.}
%Next, Bob chooses a point $x_2\in X$ such that
%$\omega_2\df(x_2,\alpha \beta\rho)\le_{s}\omega'_1$. 
%Continuing in the same manner,   one obtains a  sequence 
%\begin{center}$
%$\omega_1 \ge_{s} \ldots\ge_{s}
%\ldots \omega_k = (x_k,\rho_k)\ge_{s}\omega'_k = (x'_k,\rho'_k)\ge_{s}\omega_{k+1}= (x_{k+1},\rho_{k+1})%\ge_{s} \ldots$
%\end{center}
%satisfying
As the game is played on a complete metric space
and the diameters of the nested  balls 
\begin{center}
$B(\omega_1) \supset  \ldots\supset 
B(\omega_k) \supset B(\omega'_k) \supset \ldots$
\end{center}
tend to zero as $k\rightarrow\infty$, 
the intersection of these balls 
is a point $x_\infty\in X$. Call Alice the winner if $x_\infty\in {S}$. 
Otherwise Bob is declared the winner. 
A strategy consists of specifications for a player's choices 
of centers for his or her balls 
given the opponent's previous moves. 

If for certain $\alpha$, $\beta$ and a target set ${S}$ 
Alice has a winning strategy, i.e., 
a strategy for winning the game regardless of how well Bob plays,
we say that ${S}$ is an 
{\sl $(\alpha , \beta)$-winning set\/}.
If ${S}$ and $\alpha$ are such that ${S}$ is an $(\alpha , \beta)$-winning set 
for all $\beta$ in $(0, 1)$, %%y
we say that ${S}$ is an 
{\sl $\alpha $-winning\/} set. 
Call a set {\sl winning\/} if such an $\alpha $ exists.

Intuitively one expects winning sets to be large. Indeed, every such set is clearly dense in $X$;
moreover, under some additional assumptions on the metric space winning sets can be proved to
have positive, and even full, \hd.
%(we discuss this in greater detail in \S\ref{measures}). 
For example, the fact that a winning subset of $\br^n$ has \hd\ $n$
is due to Schmidt  \cite[Corollary 2]{S1}. Another useful result of Schmidt \cite[Theorem 2]{S1} states that
the intersection of countably many $\alpha$-winning sets 
is $\alpha$-winning. 

Schmidt himself used the machinery of the game he invented to prove that certain
subsets of $\br$ or $\br^n$ are winning, and hence have full \hd. For example, he showed
 \cite[Theorem 3]{S1} that $\bold{BA}$
is $\alpha$-winning for any $0 < \alpha \le 1/2$. The same conclusion, according to   \cite[\S 8]{S1}, holds for the sets $E_b$
defined in the introduction.
 
Now let $K$ be a closed subset of $X$.
We will say that a subset $S$ of $X$ is
{\sl $(\alpha , \beta)$-winning on $K$\/} (resp., {\sl $\alpha $-winning on $K$\/}, {\sl winning on $K$\/})
if $S\cap K$ is $(\alpha , \beta)$-winning  (resp., $\alpha $-winning,  winning) for Schmidt's
game played on the metric space $K$ with the metric induced from $(X,d)$. 
%For the rest of the paper we will take $X = \R$. 
%Note that  the partial ordering 
 %condition $(x_2,\rho_2)\le_{s}(x_1,\rho_1)$ for metric spaces $(K,d)$ is  equivalent to $B(x_2,\rho_2)\subset B(x_1,\rho_1)$, regardless of the choice of $K$.
In the present paper we let $X=\R$ and take $K$ to be the support of an absolutely decaying measure.
In other words, since the metric is induced, playing the game on $K$ amounts to choosing balls in $\R$
according to the rules of a game played on $\R$, but with an additional constraint that the
centers of all the balls lie in $K$.

\smallskip

It turns out, as was observed in \cite{F},  that the decay property \equ{axiom} is very helpful for 
playing Schmidt's game on $K$. Moreover, as demonstrated by the following proposition proved in \cite{KW2}, the decay conditions are important for  estimating the \hd\ of winning sets:
%\cite[Proposition 5.1]{KW2}.

%%d shortened and rephrased
%start (\alpha,\beta) estimates
\ignore{
for any $0 < \varepsilon,\delta < 1$, %%y
there exist $c_1, c_2 > 0$ 
%, depending only on $\epsilon$
%and $\delta$,
such that if $K$ is the support of a Federer measure $\mu$ on a metric space $X$
with constants $\varepsilon$ %%y
and $\delta$ as in Definition \ref{decay-old},  $U\subset X$ is open
with $\mu(U) > 0$, and 
$S$ is $(\alpha,\beta)$-winning on $K$ with $0 < \beta < 1/2$, then
%\begin{equation}
%\label{dim bound}
$$\dim(S\cap K\cap U) \ge
 \underline{d}_\mu(U) -
\frac{c_1|\log \alpha| +c_2}{|\log \alpha| + |\log \beta|}\,.$$
%\end{equation}
From this one can easily derive the following
}
%end (\alpha,\beta) estimates

\begin{prop} 
\label{dimbound} \cite[Proposition 5.1]{KW2}
Let $K$ be the support of a Federer measure $\mu$  on a metric space $X$, 
and let $S$ be winning on $K$. Then for any open $U\subset X$ 
with $\mu(U) > 0$ one has 
$$\dim(S\cap K\cap U) \ge
 \underline{d}_\mu(U)\,.$$
\end{prop}

In particular, in the above proposition one can replace $\underline{d}_\mu(U)$ 
with $\gamma$ if $\mu$ is $(C,\gamma)$-absolutely decaying.
Note that this generalizes  estimates for the \hd\ of winning sets due to Schmidt \cite{S1}
 for $\mu$ being Lebesgue measure on
$\R^n$, and to Fishman  \cite[\S 5]{F} for measures satisfying a power law.

\smallskip

The next lemma is another example of the absolute decay of a measure being helpful for
playing Schmidt's game on its support:

\begin{lem}
\label{log turns} 
Let $K$ be the support of a $(C,\gamma)$-absolutely decaying measure on $\R$,
and
 let $\alpha$ be as in  \equ{alpha}.
%\Kaf\ 
Then for every $0 < \rho < \rho_0$, $x_1\in K$ and $y_1,\dots,y_N \in \R$,
there exists $x'_1 \in K$ with %\comdima{Renamed points, hope it is OK.}
\eq{containment}{B(x'_1,\alpha \rho) \subset B(x_1,\rho)}
and, for at least half of the points $y_i$,
\eq{distance}{d(B(x'_1,\alpha \rho),y_i) > \alpha\rho.}
\end{lem}

\begin{proof}
%Let $\alpha = \frac14\left( \frac{1}{3C} \right)^{\frac{1}{\gamma}}$.
If $B(x_1,2\alpha\rho)$ contains not more than half of the points $y_i$,
then clearly we can take $x'_1 = x_1$.
Otherwise, $B(x_1,2\alpha\rho)$ contains at least half
of the points $y_i$. Let $x_0$ and $x_2$ be the endpoints of $B(x_1,\rho)$.
By \equ{axiom}
$$\mu\big(B(x_i,4\alpha\rho)\big) < C(4\alpha)^{\gamma}\mu\big(B(x_1,\rho)\big) \under{\equ{alpha}} < \frac{1}{3}\mu\big(B(x_1,\rho)\big)\,,$$
for $i=0,1,2$,
so there is a point $x'_1 \in K$ which is not in %%y
$B(x_i,4\alpha\rho)$ for $i=0,1,2$, and 
hence satisfies both \equ{containment} and \equ{distance} 
for all $y_i$ contained in $B(x_1,2\alpha\rho)$.
\end{proof}

We note that \equ{containment} in particular implies that $(x'_1,\alpha\rho)\le_s(x_1,\rho)$;
thus it would be a valid choice of Alice in an $(\alpha,\beta)$-game played 
on $K$ in response to $B(x_1,\rho)$ chosen by Bob. Therefore the above lemma can be  used 
to construct a winning strategy for Alice  choosing balls which stay away from some prescribed
sets of `bad' points $y_1,\dots,y_N$. This idea is motivated by the proof of Lemma 1 in \cite{M2}.

\smallskip

Furthermore, the above lemma immediately implies

\begin{cor}\label{minusone}
Let $K$ be the support of a $(C,\gamma)$-absolutely decaying measure on $\R$,
 let $\alpha$ be as in  \equ{alpha},  let $S\subset \R$ be $\alpha$-winning on $K$,
 and let $S'\subset S$ be countable. Then $S\ssm S'$ is also  $\alpha$-winning on $K$.
\end{cor}

%\comliorryan{Should we show $K\setminus \{y\}$ is $\alpha$-winning and drop
%the last sentence?}
%\comdima{Changed it -- is it OK now?}
\begin{proof} In view of the countable intersection property,
it suffices to show that $\R\ssm\{y\}$ is $(\alpha,\beta)$-winning on $K$
for any $y$ and any $\beta$. We let Alice play arbitrarily until the radius of a ball chosen by Bob is
not greater than $\rho_0$. Then apply Lemma \ref{log turns} with $N = 1$ and $y_1 = y$,
which yields a ball not containing $y$. Afterwards she can keep playing arbitrarily, winning the game.
\end{proof}

We note that such a property is demonstrated in \cite[Lemma 14]{S1} for games played on a Banach space of positive dimension.
%{\bf exact reference}.

\ignore{

In the next section we will describe a certain class of closed subsets $K$ of $\R$ on which
sets like $\bold{BA}$ or $\tilde E(\ct,y)$ discussed in the introduction
happen to be %quasi-winning, and sometimes 
winning. 
In particular it will imply that the intersections
$K\cap \tilde E(\ct,y)$
\comdima{to be changed later, if we wish.} or $K\cap \bold{BA}$  have positive, and in many cases full, \hd.

\smallskip

We close the section with a geometric lemma which will be used in the
proof of our main theorem.
}

\section{Proofs}
\label{proof}
%The proof of our main theorem 

\begin{proof}[Proof of Theorem \ref{theorem}]
Let $\alpha$ be as in %Lemma \ref{log turns} 
\equ{alpha} and let $0 < \beta < 1$. 
Suppose $K$ supports a $(C,\gamma)$-absolutely decaying measure,
$\varphi : \R\to\R$ is bi-Lipschitz,  $\ct = (t_n)$ is a sequence of positive reals
satisfying
\begin{equation}
\label{ratio}
\inf_n \frac{t_{n+1}}{t_n} = M>1\,,
\end{equation}
and $\cy = (y_n)$ is a sequence of points in $\ct$. Our goal is to specify
a strategy for Alice allowing to zoom in on $\varphi\big(\tilde E(\ct,\cy)\big)\cap K$.

Choose $N$ large enough so that 
\begin{equation}
\label{large N}
(\alpha\beta)^{-r} \leq M^N,\text{ where }r  \df \lfloor \log_2 N\rfloor +1.
\end{equation}
Here and hereafter $\lfloor \cdot \rfloor$ denotes the integer part.

%It is easy to see\comdima{Is it?} that 
Note that without loss of generality one can
replace the sequence
$\ct$ with its tail
 $\ct' \df (t_n: n \ge n_0)$;  indeed, it is easy to see that 
$$ \tilde E(\ct,\cy)
\smallsetminus \tilde E(\ct',\cy')\,,$$ where $\cy' \df (y_n: n \ge n_0)$, is at most countable;
therefore the claim follows from   Corollary \ref{minusone}.
Consequently, one can assume that\footnote{The same argument shows that the assumption of the lacunarity of $\ct$ 
in Theorem \ref{theorem} can
be weakened to {\sl eventual lacunarity\/}, that is, to $\liminf_{n\to\infty}\frac{t_{n+1}}{t_n} > 1$.} $t_n > 1$ for all $n$.

Let $L$ be a bi-Lipschitz  constant for $\varphi$; in other words, %\comdima{Added a bit more detail.}
\eq{lip}{
\frac1L \le \frac{|\varphi(x) - \varphi(y)|}{|x-y|} \le L\quad \forall \,x\ne y\in\R\,.} 
The game begins with Bob choosing $(x_1,\rho^{\prime})\in \Omega = K\times \R_+$.
Let $k_0$ be the minimal positive integer satisfying
\eq{rho small}{\rho \df (\alpha\beta)^{k_0-1}\rho^{\prime}
	< \min\left(\frac{1}{2%t
	L}(\alpha\beta)^{-r+1},\rho_0\right),} 
where $\rho_0$ is as in Definition \ref{decay}. % and $t = \min(t_1,1)$.
Alice will play arbitrarily until her $k_0$th turn.
Then $\omega_{k_0}=(x_2,\rho)$ for some $x_2\in K$.
Reindexing, set $\omega_1=\omega_{k_0}$.
Let $$c \df \frac{%t
\rho}{L}(\alpha\beta)^{3r}.$$
For an arbitrary $k \in \N$, define 
$$
I_k \df \{n\in\N :(\alpha\beta)^{-r (k-1)} \leq t_n 
	< %t
	(\alpha \beta)^{-r k}\};
	$$
	note that $\#I_k \le N$ in view of (\ref{ratio}) and (\ref{large N}).
	
	 Our goal now is to describe Alice's strategy for choosing $\omega_i'\in\Omega$, $i\in\N$, 
%$\omega_{r(k+1)}',\dots, \omega_{r(k+2)-1}'\in\Omega$
to ensure that for any $k\in\N$, 
\eq{goal}{d\left(\pi\big(t_n\varphi^{-1}(x)\big),y_n\right) \geq c   
%%y I think, the inequality should be a large one.
\ \text{ whenever  }\ 
x \in B(\omega'_{r(k+2) - 1})\ \text{  and }\ n\in I_k\,.}
%\comdima{Changed from $\omega$ to $\omega'$  here -- isn't it better this way?} 
%and %t
% $n$ is such that 
%\eq{Npoints}{(\alpha\beta)^{-r (k-1)} \leq t_n 
%	< %t
%	(\alpha \beta)^{-r k}\,.}
Then if we let
 $$x_\infty \df \bigcap_i B(\omega^{\prime}_i) =  \bigcap_k B(\omega'_{r(k+2) - 1})\,,$$
 which is clearly an element of $K$, we will have $\varphi^{-1}(x_\infty)
 \in  \tilde E(\ct,\cy)$; in other words, \equ{goal} enforces that $x_\infty\in\varphi\big(\tilde E(\ct,\cy)\big)\cap K$, as required.

\smallskip
To achieve \equ{goal}, Alice may choose $\omega_i'$ arbitrarily for $i < 2r$.
Now fix  $k\in\N$ and observe that whenever $n\in I_k $ and 
 $m_1 \ne m_2 \in \Z$, one has
$$\left|\frac{y_n+m_1}{t_n} -\frac{y_n+m_2}{t_n} \right| \geq t_n^{-1} 
	> %\frac{1}{t}
	(\alpha \beta)^{r k},$$
so, by \equ{lip},
\eq{preimages}{\left| \varphi \left(\frac{y_n+m_1}{t_n} \right) - 
	\varphi \left( \frac{y_n+m_2}{t_n} \right) \right|
	> %\frac{1}{t}
	\frac{1}{L}(\alpha \beta)^{r k}.}
Because of \equ{rho small}, the %radius 
diameter of $B(\omega_{r (k+1)})$ is
$$
2(\alpha\beta)^{r(k+1)-1}\rho < \frac{1}{%2%t
L}(\alpha\beta)^{rk},
$$
so by \equ{preimages} the set
$$Z 
\df  \left\{\varphi\left(\frac{y_n+m}{t_n} \right) : m \in \Z, \ 
	%t
	n\in I_k\right\}
$$
has at most $N$ elements in $B(\omega_{r (k+1)})$.
Applying Lemma \ref{log turns} $r$ times, 
Alice can choose $\omega_{r(k+1)}',\dots, \omega_{r(k+2)-1}'\in\Omega$ in such a way that
$$d\big(B(\omega'_{r(k+2) - 1}),Z\big)
	\geq (\alpha\beta)^{r(k+2)}\rho\,.$$
 Therefore, again by \equ{lip}, for any $x\in B(\omega'_{r(k+2) - 1})$,  $m \in \Z$
	and
	$n\in I_k$
 one has  
 % so whenever %$%t
%(\alpha\beta)^{-r (k-1)} \leq t_n 
%	<% t
%	(\alpha \beta)^{-r k}$
%\equ{Npoints} holds, one has
$$\big|t_n\varphi^{-1}(x)-(y_n+m)\big| \geq \frac{t_n}{L}\left|x - \varphi\left(\frac{y_n+m}{t_n} \right)\right|
 \geq \frac{t_n}{L}
	(\alpha\beta)^{r(k+2)}\rho
	\geq \frac{%t
	\rho}{L}(\alpha\beta)^{3r} = c\,,$$
%Hence,
%$\bigcap_k B_K(\omega^{\prime}_k) \in \varphi\big(\tilde E(\ct,\cy)\big)\cap K$.
which implies \equ{goal}.
\end{proof}

Recall that it was shown in \cite{S1} that   $\bold{BA}$ is a winning subset of $\R$.
In \cite{F}, this set, and its nonsingular affine images, 
was shown to be $\alpha$-winning on the support of 
any $(C,\gamma)$-absolutely decaying measure on $\R$, where   $\alpha$ depends only on
$C$ and $\gamma$. In what follows we prove a slight generalization of this result for
bi-Lipschitz images. The technique used is similar to the one used in the proof of the main theorem.
We include it for the sake of completeness.

\begin{thm}\label{ba}
Let $K$ be the support of a $(C,\gamma)$-absolutely decaying measure on $\R$,
and let $\alpha$ be as in \equ{alpha}. Then
for every bi-Lipschitz map $\varphi : \R\to \R$, the set
$\varphi(\bold{BA})$ is $\alpha$-winning on $K$.
\end{thm}

\begin{proof}
Again, take %$\alpha$ to be as in Lemma \ref{log turns} and 
an arbitrary $0 < \beta < 1$,
and let  $L$ be as in \equ{lip}.
%, so that $\varphi^{-1}$ is 
%also bi-Lipschitz with constant $L$.
Let $R= (\alpha\beta)^{-\frac{1}{2}}$. 
The game begins with Bob choosing $(x_1,\rho^{\prime})\in \Omega$.
Let $k_0$ be the minimal positive integer satisfying 
\eq{rho}{(\alpha\beta)^{k_0-1}\rho^{\prime}<\min\left(\frac{\alpha\beta}{2L}, \rho_0\right)\,,} where
$\rho_0$ is as in Definition \ref{decay},
and denote $\rho \df(\alpha\beta)^{k_0-1}\rho {'}$.
Alice will play arbitrarily until her $k_0$th turn.
Then $\omega_{k_0}=(x_2,\rho)$ for some $x_2\in K$.
Reindexing, set $\omega_1=\omega_{k_0}$.
Let $c = \frac{R^2\alpha\rho}{L}$.

Fix an arbitrary $k\in\N$.
We will describe Alice's strategy for choosing $\omega_k'$
such that
\begin{equation}
\label{W_k2}
\left| \varphi^{-1}(x) - \frac{p}{q}\right| > \frac{c}{q^2}\text{ for all } x\in B(\omega'_k), R^{k-1} \le q < R^{k}.
\end{equation}
Clearly the existence of such strategy implies that she can play so that $\bigcap_k B(\omega'_k)$ lies in 
 $K\cap \varphi(\bold{BA})$. 
 
Note that for any distinct $\frac{p_1}{q_1}, \frac{p_2}{q_2}\in \R$
with $R^{k-1} \le q_1, q_2 < R^k$, 
$$
\left| \frac{p_1}{q_1}-\frac{p_2}{q_2}\right|
= \left| \frac{p_1q_2-p_2q_1}{q_1q_2}\right| > \frac{1}{R^{2k}}.
$$
Hence, $\left|\varphi\left(\frac{p_1}{q_1}\right) - \varphi\left(\frac{p_2}{q_2}\right) \right| 
\ge \frac{1}{L}R^{-2k}$. 
But
$$\text{diam}\big( B(\omega_{k}) \big)\le 2\rho(\alpha\beta)^{k-1} \under{\equ{rho}}< \frac{1}{L}R^{-2k}\,,$$
so $B(\omega_{k})$ contains at most one
point $\varphi\big(\frac{p}{q}\big)$ with $R^{k-1} \le q < R^k$.
In view of  Lemma \ref{log turns}, where we put $N = 1$, 
Alice can choose $\omega'_{k}\in \Omega$ such that, for every $x \in B(\omega'_{k})$
and $(p,q) \in \mathbb{Z}\times \N$ with $R^{k-1} \le q < R^k$,
one has
$$
\left| x - \varphi\left(\frac{p}{q}\right)\right| > \alpha\rho(\alpha\beta)^k
= \alpha\rho R^{-2k} > \frac{R^2\alpha\rho}{q^2}\,.
$$
Again by \equ{lip}, we obtain 
$$
\left|\varphi^{-1}(x) - \frac{p}{q}\right| 
> \frac{R^2\alpha\rho}{Lq^2} = \frac{c}{q^2}\,,
$$
and (\ref{W_k2}) is established.
\end{proof}

As an immediate consequence of 
Proposition \ref{dimbound} and
the countable intersection property 
of winning sets,
we obtain the following

%%d added U
\begin{cor}
\label{cor} 
Let $K$ be the support of a $(C,\gamma)$-absolutely decaying measure on $\R$,
and let $\alpha$ be as in \equ{alpha}. 
Then
given lacunary sequences $\ct_k$, sequences $\cy_k\in \T$,
bi-Lipschitz maps $\varphi_k,\psi_k: \R\to \R$, and an open set $U \subset \R$
with % $\mu(U) > 0$ %%d
$U\cap K \ne\varnothing$, 
one has
%\eq{finite2}{
$$\dim\left( \bigcap_{k = 1}^{\infty} K \cap U \cap \varphi_k(\bold{BA}) \cap \psi_k\big( \tilde E(\ct_k, \cy_k)\big)\right) \ge \gamma\,.$$%} %%d removed the tag
\end{cor}

In particular we can have $\gamma = \dim(K)$ when the measure satisfies a power law (e.g.\ 
when $K$ is equal to $\R$ or to the
middle third Cantor set).

\smallskip

We conclude the section with an application of Theorem \ref{theorem} to affine expanding maps $f_{b,c}$ as defined in \equ{deffbc}:

\begin{cor}
\label{affine}
Let $K$ be the support of a $(C,\gamma)$-absolutely decaying measure on $\R$,
and let $\alpha$ be as in \equ{alpha}. 
 Then
 for every bi-Lipschitz map $\varphi: \R \to \R$, 
$b\in\Z_{\ge 2}$ and $c,y\in \T$,  the set $\varphi\left(\pi^{-1}\big( E(f_{b,c}, y)\big)\right)$ is $\alpha$-winning on $ K$. 
\end{cor}

\begin{proof} Since $f_{b,c}$ is a composition of $f_b$ with an isometry of $\T$, 
it is easy to construct a sequence 
%%y \comdima{Is it clear or should we write down this sequence?}  Clear enough, I think!
of points $\cy = (y_n)$ of $\T$ such that, 
with $\ct = (b^n)$, %%y
one has 
$x\in \tilde E(\ct,\cy)$ if and only if $\pi(x)\in E(f_{b,c}, y)$.
\end{proof}

\section{Applications, related results and further questions}

%%y

\subsection{Trajectories avoiding intervals}  
Recently   a quantitative modification of Schmidt's proof of
abundance
of numbers normal to no base was introduced in the work of R.\ Akhunzhanov.  To describe it, let us define 
$$\hat E(b,A) = \bigcap_{y\in A} \tilde E(b,y) =  \big\{ x\in\R: A\cap \overline{\{\pi(b^n x) : n\in \N\}} = \varnothing
\big\} $$
for a subset $A$ of $\T$.
Clearly when $A = B(0,\delta)$ is a $\delta$-neighborhood of $0$ in $\T$, every number $x\in \hat E(b,A)$ has a uniform (depending on $\delta$) %%y
upper bound on the number of consecutive zeros in the $b$-ary expansion.
It is easy to see that whenever $A$ contains an interval,  $\hat E(b,A)$  is nowhere dense and has positive 
Hausdorff codimension. 
Nevertheless it was proved in \cite{A1, A2} that for any $\vre > 0$ 
and any integer $b\ge 2$ there exists a %%y
positive (explicitly constructed) $\delta = \delta_{b,\vre}$ such that the set
$$
\bigcap_{b \in\Z_{\ge 2}}\hat E\big(b,B(0,\delta_{b,\vre})\big)$$
has \hd\ at least $1-\vre$. The proof is based on Schmidt's game, namely on so-called {\sl $(\alpha,\beta,\rho)$-winning sets\/} of the game. This technique readily extends to playing on supports of absolutely decaying measures. Namely, one can show that given $C,\gamma,\vre > 0$ and integer $b\ge 2$,
there exists $\delta = \delta_{C,\gamma, b,\vre}$ such that 
$$
\dim\left(
\bigcap_{b \in\Z_{\ge 2}}K \cap \hat E\big(b,B(0,\delta_{C,\gamma,b,\vre})\big)\right) > \gamma - \vre$$
whenever $K$ supports a $(C,\gamma)$-absolutely decaying measure. Details will 
be described elsewhere.

\subsection{Are these sets null?}
 It is not hard to construct examples
of  absolutely decaying measures $\mu$  such that 
%%y \comdima{This is just a draft of a remark that we might want to make along these lines, 
%%y emphasizing that 
%%y the sets we proved to be winning can sometimes be small in the sense of measure.}
%%y O.K. with the remark!
$K = \supp\,\mu$ lies entirely
inside a set of the form $\tilde E(b,y)$ for some $b\in\Z_{\ge 2}$, or inside
%${\bf BA}$.
the set of \ba\ numbers. However in many cases, under some additional assumptions on $\mu$ 
one can show that those sets, proved to be winning on $K$ in the present paper,
have measure zero. For example, it is proved in \cite{C} that almost all $x$
in the middle third Cantor set,  %%y
with respect to the coin-flipping measure, are normal to base $b$ whenever
$b$ is not a power of $3$. And in a recent work \cite{EFS} of M.\ Einsiedler, U.\ Shapira and 
the third-named author it is established that $\mu({\bf BA}) = 0$ whenever $\mu$ is
$f_b$-invariant for some  $b\in\Z_{\ge 2}$ and has positive dimension. It seems interesting to
ask for general conditions on a measure on $\R$,  possibly stated in terms of invariance
under some dynamical system, which guarantee that whenever $y\in \T$, 
sets  $\tilde E(b,y)$ for a fixed $b \ge 2$  %%y
%, or, more generally,
 %$\tilde E(\ct,y)$ for a fixed lacunary sequence $\ct$, 
 have measure zero.

\subsection{Strong winning sets} In a recent preprint \cite{Mc} C.\ McMullen introduced a
 modification of Schmidt's game, where
condition \equ{balls} is replaced by \eq{strongballs}{\rho_k' \geq \alpha \rho_k\text{ and }\rho_{k+1} \geq \beta \rho_k'\,,}
and $S\subset X$
is said to be {\sl $(\alpha,\beta)$-strong winning\/} if Alice
 has a winning strategy in the game deÞned by  \equ{strongballs}.
Analogously, we define $\alpha$-strong winning and strong winning sets.  %%y
It is straightforward to verify that
$(\alpha,\beta)$-strong winning implies $(\alpha,\beta)$-winning, and that a countable intersection of $\alpha$-strong winning sets is $\alpha$-strong winning. Furthermore, this class has stronger invariance properties,
e.g.\ it is proved in \cite{Mc} that strong winning subsets of $\R^n$ are preserved by quasisymmetric 
homeomorphisms. 
McMullen notes that many examples of winning sets arising naturally in dynamics and Diophantine approximation seem to  also be strong 
winning. The sets considered in this paper are no exception: it is not hard to modify
our proofs to show that, under the assumptions of Theorems \ref{theorem} and \ref{ba},
the sets $\tilde E(\ct,\cy)$ and $\bf BA$ are $\alpha$-strong winning on $K$.

\subsection{More general self-maps of $\T$}\label{nonlinear} It would be interesting to unify Theorems \ref{theorem}
%%y \comdima{Feel free to edit, it is just a draft.}  O.K. also!
and \ref{ba} by describing a class of maps $f:\T\to\T$
%, containing both the Gauss map and
 for which one can prove sets of the form
$E(f,y)$ to be winning on $K$ whenever $K\subset \T$ supports an absolutely decaying measure. 
An important special case is a map $f$ given by multiplication by $b$ when $b > 1$ is not an integer;
that is, constructed by identifying $\T$ with $[0,1)$ and defining $f(x) = bx \mod 1$. With this definition,
the set \equ{defebut} does not coincide with the $\pi$-preimage of $E(f,y)$, and the methods
of the present paper do not seem to yield any information. 
%In particular it would be interesting to
%find out
%whether  $E(f,0)$ is a winning subset of $\T$, e.g.\ when $b = 3/2$.
Some results along these lines have been obtained recently in \cite{Fae, FPS}.

\subsection{A generalization to higher dimensions}\label{matrices} 
The method developed in the present paper has been  extended in \cite{BFK}  to a multi-dimensional 
set-up, that is, with a lacunary sequence of real numbers acting on $\R$
replaced by a sequence of $m\times n$ matrices, whose operator norms form a lacunary sequence,
acting on $\R^n$. This, among other things, generalizes a result of Dani \cite{D}
on orbits of toral endomorphisms.  A higher-dimensional analogue of Theorem \ref{theorem} can
be established for absolutely decaying measures on $\R^n$.
% which are both absolutely decaying and Federer.
%(Those %classes of measures on $\R^n$ 
%have been
%extensively studied in \cite{KLW}, and referred to as {\sl absolutely friendly}
%in \cite{PV}.) 
Note that the  definition of absolutely decaying measures on $\R^n$ \cite{KLW}
is the same as Definition \ref{decay} but with balls $B(y,\vre\rho)$ being replaced by $\vre\rho$-neighborhoods
of affine hyperplanes. Also,   Proposition \ref{equiv-af} does not extend to $n > 1$, that is, absolute decay
does not imply Federer, and a combination of 
efd and Federer does not imply absolute decay. %More detail  to appear in a forthcoming paper \cite{BFK}.

\bibliographystyle{alpha}

\begin{thebibliography}{99}

\bibitem%[A1]
{A1}
R. K. Akhunzhanov, 
\textsl{On nonnormal numbers}, Mat.\ Zametki  {\bf 72}  (2002),  150--152 
(in Russian); translation in  Math.\  Notes  {\bf 72}  (2002), 135--137.

\bibitem%[A2]
{A2}
\bysame, 
\textsl{On the distribution modulo $1$ of exponential sequences},
Mat.\ Zametki  {\bf 76}  (2004), 163--171 (in Russian);  
translation in  Math.\ Notes  {\bf 76}  (2004),  153--160.
 
 

%\bibitem[BD]{BD}
%V.\,I.\ Bernik and M.\,M.\ Dodson, \textsl{Metric {D}iophantine
%approximation on manifolds}, Cambridge University Press,
%Cambridge, 1999.


\bibitem%[BFK]
{BFK}
R.\ Broderick, L.\ Fishman and D.\ Kleinbock, \textsl{Schmidt's game, fractals, and orbits of toral endomorphisms}, 
Preprint, {\tt arXiv:1001.0318}.


\bibitem%[BHKV]
{BHKV}   %%y
Y.\ Bugeaud, S.\ Harrap, S.\ Kristensen and S.\ Velani,   %%y
\textsl{On shrinking targets for $\Z^m$ actions on tori}, Preprint, {\tt arXiv:0807.3863}.   %%y

\bibitem%[C]
{C}J.W.S.\ Cassels, \textsl{On a problem of Steinhaus about normal numbers}, 
Colloq.\ Math.\ {\bf  7}  (1959), 95--101.

\bibitem%[D]
{D}S.G.\ Dani,  \textsl{On orbits of endomorphisms of tori and the Schmidt game}, 
Ergod.\ Theory  Dynam.\ Systems {\bf 8} (1988), 523--529.


%\bibitem[D2]{Dani survey} \bysame, 
% \textsl{On badly approximable numbers, Schmidt games and bounded orbits
%of flows}, in:  {\bf Number theory and dynamical systems (York, 1987)},  69--86,
%London Math.\ Soc.\ Lecture Note {\bf 134} Cambridge Univ.\ Press, 1989.
 
 \bibitem%[Do]
 {Do}D.\ Dolgopyat,  \textsl{Bounded orbits of Anosov flows}, 
Duke Math.\ J.\ {\bf 87} (1997), no.\ 1, 87--114. 

\bibitem%[EFS]
{EFS} M.\ Einsiedler, L.\ Fishman and U.\ Shapira, \textsl{\da\  on fractals}, Preprint, {\tt arXiv:0908.2350}.

   

\bibitem%[Fa]
{Fa} K.\ Falconer, \textsl{Fractal
geometry. Mathematical foundations and applications},  John Wiley \&
Sons, Inc., Hoboken, NJ, 2003. 


\bibitem%[F\"a]
{Fae} D.\  F\"arm, \textsl{Simultaneously Non-dense Orbits Under Different Expanding Maps}, 
Preprint, {\tt	arXiv:0904.4365v1}.


\bibitem%[FPS]
{FPS} D.\ F\"arm, T.\ Persson and J.\ Schmeling, \textsl{Dimension of Countable Intersections of Some Sets Arising in Expansions in Non-Integer Bases}, Preprint (2009).


\bibitem%[Fi]
{F} L.\ Fishman, \textsl{Schmidt's game on fractals},  Israel J.\ Math.\
{\bf 171}  (2009), no.\ 1, 77--92.

\bibitem%[J]
{J} V.\ Jarn\'ik, \textsl{Zur metrischen Theorie der Diophantischen Approximationen},
Prace Math-fiz.\ {\bf 36} 2.\ Heft (1928).



\bibitem%[K]
{K}D.\ Kleinbock, \textsl{Nondense orbits of flows on homogeneous spaces},
Ergodic Theory Dynam.\ Systems {\bf 18} (1998), 373-396.


%\bibitem[K2]{dima pamq}   \bysame,  \textsl{Diophantine exponents of measures and homogeneous dynamics},
%Pure Appl.\ Math.\ Q.\ {\bf 4} (2008), 81--97.

%\bibitem[K3]{dichotomy}   \bysame,  \textsl{An `almost all versus no' dichotomy in homogeneous dynamics and Diophantine approximation}, Geom.\ Dedicata, to appear,
%{\tt arXiv:0904.1614}.


\bibitem%[KLW]
{KLW}D. Kleinbock, E. Lindenstrauss and B. Weiss, 
\textsl{On fractal measures and diophantine approximation}, Selecta Math.\ 
{\bf 10} (2004), 479--523.


%\bibitem[KM]{KM}  D.\ Kleinbock and G.\,A.\
%Margulis,   \textsl{Flows on homogeneous spaces and Diophantine
%approximation on manifolds}, Ann.\ Math.\  {\bf 148} (1998),
%339--360.


\bibitem%[KW1]
{KW1}D.\ Kleinbock and B.\ Weiss, 
\textsl{Badly approximable vectors on fractals}, 
Israel J.\ Math.\  {\bf 149} (2005), 137--170.


\bibitem%[KW2]
{KW2}\bysame, 
\textsl{Modified Schmidt games and Diophantine approximation with weights}, 
Advances in Mathematics {\bf 223} (2010),  1276--1298.

 
\bibitem%[KTV]
{KTV}S.\ Kristensen, R.\ Thorn, S.L.\ Velani, 
\textsl{Diophantine approximation and badly approximable sets}, 
Advances in Math.\ {\bf 203} (2006), 132--169.


\bibitem%[dM]
{Ma} B.\ de Mathan, \textsl{Numbers contravening a condition in
density modulo $1$},  Acta Math.\ Acad.\ Sci.\ Hungar.\
{\bf 36} (1980), 237--241.


\bibitem%[MU]
{MU}D.\ Mauldin and M.\ Urbanski, \textsl{ The doubling property
of conformal measures of infinite iterated function systems}, 
J.\ Number Th.\ {\bf 102} (2003), 23--40.


\bibitem%[Mc]
{Mc}C.\ McMullen, \textsl{Winning sets, quasiconformal maps and 
Diophantine approximation}, Preprint (2010).

\bibitem%[M1]
{M1} N.G.\ Moshchevitin, \textsl{Sublacunary sequences and winning
sets},  Mat. Zametki  {\bf 77}  (2005),  no.\ 6, 803--813 (in Russian);  translation in Math.\ Notes
{\bf 78} (2005), no.\ 4, 592--596.

\bibitem%[M2]
{M2} \bysame, \textsl{A note on badly approximable affine forms
and winning sets}, Preprint (2008), {\tt arXiv:0812.3998v2}.

 
 \bibitem%[P1]
{P1}A.D.\ Pollington,
\textsl{On nowhere dense $\Theta$-sets}, 
Groupe de travail d'analyse ultramŽtrique.\ 
{\bf 10} (1982-1983), no.\ 2, Exp. No. 22, 2 p. 

\bibitem%[P2]
{P2} \bysame, 
\textsl{On the density of sequence $\{\eta_{k}\xi\}$},
Illinois J.\ Math.\ 
{\bf 23} (1979), no.\ 4,  511--515.


%\bibitem[PV]{PV}A.D.\ Pollington and S.L.\ Velani, 
%\textsl{Metric Diophantine approximation and `absolutely friendly' measures},
%Selecta Math.\ 
%{\bf 11} (2005) 297--307. 

\bibitem%[S]
{S1}W.M.\ Schmidt, \textsl{On badly approximable numbers and certain games}, 
Trans.\ A.M.S.\ {\bf 123} (1966), 27--50.

\bibitem%[SU]
{SU} B.\ Stratmann and M.\ Urbanski, \textsl{Diophantine extremality of the Patterson measure},  Math.\ Proc.\ Cambridge Phil.\ Soc.\ {\bf 140} (2006), 297--304.

\bibitem%[T]
{Ts} J.\ Tseng, \textsl{Schmidt games and Markov partitions},  Nonlinearity  {\bf 22}  (2009),  no.\ 3, 525--543.

\bibitem%[U1]
{Urbanski} M.\ Urbanski, \textsl{The Hausdorff dimension of
the set of points with non-dense orbit under a hyperbolic dynamical
system,} Nonlinearity {\bf 4} (1991), 385--397.

\bibitem%[U2]
{U2} \bysame, \textsl{Diophantine approximation of self-conformal measures},  J.\ Number Th.\ {\bf 110} (2005), 219--235.

\bibitem%[U3]
{U3} \bysame, \textsl{Diophantine approximation of conformal measures of one-dimensional iterated function systems}, 
Compositio Math.\ {\bf 141} (2005),  869--886.

\bibitem%[U4]
{U4} \bysame, \textsl{Finer Diophantine and regularity properties of $1$-dimensional parabolic IFS},
Real Anal.\ Exchange {\bf 31} (2005/06), no.\ 1, 143--163. 


\bibitem%[V]
{Veech} W.A.\ Veech, {\em Measures supported on the
set of uniquely ergodic directions of an arbitrary holomorphic
1-form}, Ergodic Theory Dynam.\ Systems  {\bf 19} (1999), 1093--1109.

\bibitem%[W]
{W} B.\ Weiss, \textsl{Almost no points on a Cantor set are very well approximable}, Proc.\ R.\ Soc.\ 
Lond.\ A  {\bf 457} (2001),  949--952.







\end{thebibliography}

\end{document}